%% file: main.tex
\documentclass[11pt, letterpaper, leqno, twoside]{article}
\title{\textsc{\textbf{{Parabolic frequency on Ricci flows}}}}

\author{\textsc{Julius Baldauf\thanks{Supported in part by the National Science Foundation. {\it E-mail}: \texttt{juliusbl@mit.edu}} \qquad \quad \qquad Dain Kim}
\vspace{0.2cm}\\

    \textsc{\footnotesize MIT Department of Mathematics}\vspace{-0.1cm}\\
    \textsc{\footnotesize Cambridge, MA} \vspace{-0.05cm}
}
\date{}

\newcommand{\Addresses}{{% additional braces for segregating \footnotesize
  \bigskip
  \bigskip
  \small
    \textsc{Department of Mathematics}\par\nopagebreak
    \textsc{Massachusetts Institute of Technology}\par\nopagebreak
    \textsc{77 Massachusetts Avenue} \par\nopagebreak
    \textsc{Cambridge, MA 02139} 
    
    \medskip
    \medskip
    
    \textit{Correspondence to be sent to:} \texttt{juliusbl@mit.edu}

}}

\usepackage[shadow,textwidth=width, color=green!20, textsize=tiny]{todonotes}
\reversemarginpar

\usepackage{amsmath}
\usepackage{amsthm,xcolor}
\usepackage[colorlinks, linkcolor=red]{hyperref}
\usepackage[top=1.5in, lmargin=1in, rmargin=1in]{geometry}
\usepackage[english]{babel}
\usepackage{amssymb}
\usepackage{mathrsfs}
\usepackage{mathtools}
\usepackage{bbm}
\usepackage{amsthm}
\usepackage{fancyhdr}
\usepackage{tikz-cd}
\usepackage{comment}
\usepackage{etoolbox}
\usepackage{breqn}
\usepackage{esint}
\usepackage{appendix}

\usepackage[T1]{fontenc}
\usepackage[center]{titlesec}
\usepackage{xcolor}
\usepackage[bottom]{footmisc}
\usepackage{indentfirst}
\usepackage{xpatch}

\hypersetup{
    colorlinks=true,
    linkcolor=blue,
    filecolor=magenta,      
    urlcolor=cyan,
	citecolor=blue
}

\makeatletter
\renewcommand\th@plain{\slshape}
\xpatchcmd{\proof}{\itshape}{\slshape}{}{}
\makeatother

\makeatletter
\renewcommand\th@plain{\slshape}
\makeatother

\titlelabel{\thetitle.\quad}
\titleformat*{\section}{\centering\large\scshape\sffamily}
\titleformat*{\subsection}{\centering\normalsize\scshape\sffamily}
\titleformat*{\subsubsection}{\centering\large\scshape\sffamily}

\fancypagestyle{mypagestyle}{%
  \fancyhf{}% Clear header/footer
  \fancyhead[OC]{\small Julius Baldauf \qquad \qquad \qquad \qquad Dain Kim}% Author on Odd page, Centred
  \fancyhead[EC]{\small Parabolic frequency on Ricci flows}% Title on Even page, Centred
  \fancyhead[EL]{\thepage}%
  \fancyhead[OR]{\thepage}%
  \setlength{\headheight}{13.6pt}
}
\numberwithin{equation}{section} 

\theoremstyle{plain} 
\newtheorem{lemma}[equation]{Lemma}

\newtheorem{proposition}[equation]{Proposition}
\newtheorem{theorem}[equation]{Theorem}
\newtheorem{corollary}[equation]{Corollary}

\theoremstyle{definition}

\newcommand{\R}{\mathbb{R}}
\newcommand{\Z}{\mathbb{Z}}
\newcommand{\N}{\mathbb{N}}

\newcommand{\Div}{\mathrm{div}}
\newcommand{\be}{\begin{equation}}
\newcommand{\ee}{\end{equation}}

\newcommand{\Ric}{\mathrm{Ric}}

\newcommand{\Scal}{R}

\renewcommand{\phi}{\varphi}
\renewcommand{\epsilon}{\varepsilon}
\newcommand{\Spec}{\mathrm{Spec}}

\newcommand{\Hess}{\mathrm{Hess}}

\newcommand{\euc}{\mathrm{euc}}
\newcommand{\dL}{\mathcal{L}}

\begin{document}

\maketitle
%\tableofcontents

%%%%%%%%%%%%%%%%%%%%%%%%%%%%%%%%%%
\begin{abstract}
\input{00abstract}
\end{abstract}

%%%%%%%%%%%%%%%%%%%%%%%%%%%%%%%%%%
\input{0introduction}
%%%%%%%%%%%%%%%%%%%%%%%%%%%%%%%%%%
\section*{\small\bf Acknowledgements}
\input{01acknowledgement}
%%%%%%%%%%%%%%%%%%%%%%%%%%%%%%%%%%
\input{1section}

\input{2section}

\input{3appendix}

\input{4references}

\end{document}

%% file: 00abstract.tex
This paper defines a parabolic frequency for solutions of the heat equation on a Ricci flow and proves its monotonicity along the flow. Frequency monotonicity is known to have many useful consequences; here it is shown to provide a simple proof of backwards uniqueness. For solutions of more general parabolic equations on a Ricci flow, this paper provides bounds on the derivative of the frequency, which similarly imply backwards uniqueness.

%% file: 0introduction.tex
%%%%%%%%%%%%%%%%%%%%%%%%%%%%%%%%%%
\section{Introduction}

Growth bounds for solutions of partial differential equations provide vital information and have many useful consequences. For a harmonic function $u$ on $\R^n$, Almgren \cite{A} first proved the monotonicity of the frequency function
\begin{equation}
    N(r)
    =\frac{r\|\nabla u\|_{L^2(B_r(p))}^2}{\|u\|_{L^2(\partial B_r(p))}^2}.
\end{equation}
This function measures the rate of growth of $u$ near the point $p$. The name \emph{frequency} stems from the fact that a harmonic function which is homogeneous of degree $k$ has $N(r)\equiv k$. 

Since its introduction by Almgren, the elliptic frequency has been considerably generalized.
Garofalo-Lin \cite{GL1} extended it to harmonic functions on Riemannian manifolds, as well as to solutions of more general elliptic equations.
For a harmonic function on a Riemannian manifold, there exist constants $\Lambda,R>0$ such that $e^{\Lambda r^2}N(r)$ is monotone increasing for $r\in (0,R)$ \cite{GL1, M}.
% see also \cite[Thm.\ 2.2]{M}. 

Monotonicity of the frequency has been applied fruitfully in various contexts; for example, to study the nodal and critical sets of solutions to elliptic and parabolic equations and to prove unique continuation \cite{GL1, GL2, HL, HHL, Li, E, EFV, Lo}. 
Additionally, Colding-Minicozzi used frequency monotonicity to prove finite dimensionality of the space of polynomial growth harmonic functions on manifolds with nonnegative Ricci curvature and Euclidean volume growth \cite{CM1}. 
More recently, Taubes defined and studied versions of the elliptic frequency in the context of gauge theory \cite{T}. 
%Frequency monotonicity can also be viewed as a certain entropy monotonicity in a statistical ensemble \cite[p. 568]{Z}.

The parabolic frequency generalizes Almgren's original elliptic frequency for harmonic functions. On static manifolds, the parabolic frequency was first defined by Poon \cite{P}. 
Poon's frequency monotonicity theorem holds for manifolds with non-negative sectional curvature and parallel Ricci curvature, which are the assumptions necessary for Hamilton's matrix Harnack inequality to hold \cite{H1,H2}. 
Using the drift Laplacian, \cite{CM3} circumvented using Hamilton's matrix Harnack inequality and proved monotonicity of the parabolic frequency for arbitrary static manifolds. 

This paper is the first to study notions of parabolic frequency on general evolving manifolds, having previously been investigated only in two special contexts: on 2-dimensional Ricci flows by Li-Wang \cite{LW}, and on shrinking Ricci solitons of arbitrary dimension by Colding-Minicozzi \cite{CM3}. This paper's definition of frequency (\ref{eqn: frequency defn}) generalizes Colding-Minicozzi's to arbitrary Ricci flows. 

Let $(M^n,g(t))_{t \in I}$ be a Ricci flow, let $\tau$ denote backwards time, and let $K=(4\pi \tau)^{-\frac{n}{2}}e^{-f}$
be the conjugate heat kernel centered at some space-time point (see Section \ref{sec: monotonicity for heat operator} for precise definitions). 
Assume that along the flow there holds a Bakry-\'Emery bound $\Ric_f:=\Ric+\Hess_f\leq \frac{\kappa}{2\tau}g$, for some time-dependent constant $\kappa=\kappa(t)$. Such a bound always exists when $M^n$ is compact, for example.
This paper's definition of the parabolic frequency employs the conjugate heat kernel measure $d\nu=K\,dV$, as well as the associated drift Laplacian $\dL_f$, which is the operator defined by
\begin{equation}
    \dL_{f}u=\Delta u-\langle \nabla f,\nabla u\rangle = e^{f}\,\Div\,(e^{-f}\nabla u).
\end{equation}

The \emph{parabolic frequency} $U(t)$ of a solution $u$ of the heat equation along the Ricci flow, with\footnote{Some growth assumption is necessary to rule out the classical Tychonoff example \cite[Chapter 7]{J}.} 
$u,\partial_tu \in W^{2,2}(d\nu)$ for each time $t\in I$, 
is defined by
\begin{equation} \label{eqn: frequency defn}
    U(t)
    =-\frac{\tau\|\nabla u\|_{L^2(d\nu)}^2}{\|u\|_{L^2(d\nu)}^2}e^{-\int \frac{1-\kappa}{\tau}}
    =\frac{\tau\langle \dL_f u,u\rangle_{L^2(d\nu)}}{\|u\|_{L^2(d\nu)}^2}e^{-\int \frac{1-\kappa}{\tau}}.
\end{equation}
The exponential involving $\kappa$ is a correction term depending on the geometry of the flow; 
it is the parabolic analogue of the error term  $e^{\Lambda r^2}$ appearing in the above elliptic frequency.
By further analogy with the elliptic case, a caloric polynomial of degree $k$ has $U(t)\equiv -\frac{k}{2}$; see Appendix \ref{sec: examples}.

\begin{theorem}
[Frequency monotonicity]\label{thm: monotonicity of frequency on RFs}
The frequency $U$ is monotone increasing along the Ricci flow, and $U'(t)=0$ only if $u$ is an eigenfunction of $\dL_f$ satisfying $\dL_fu=c(t)u$, where $c(t)=e^{-\int \frac{1-\kappa}{\tau}}\frac{U}{\tau}$.
\end{theorem}

Part of the strength of frequency monotonicity is the simplicity with which it can be applied. 

\begin{corollary}
[Backwards uniqueness]\label{cor: backwards uniqueness}
If $a<b$ and $u(\cdot,b)=0$, then $u\equiv 0$ for all $t\in [a,b]$.
\end{corollary}

The results of this paper also extend to more general parabolic equations. In this case, though $U$ need no longer be monotone, its derivative can be bounded suitably to imply backwards uniqueness.

\begin{theorem}
[More general heat operators]\label{thm: more general heat operators}
Let $u:M\times [a,b]\to \R$ satisfy $|(\partial_t-\Delta) u|\leq C(t)(|\nabla u|+|u|)$ along the Ricci flow $(M^n,g(t))_{t \in I}$. If $u(\cdot,b)=0$, then $u\equiv 0$ for all $t\in [a,b]$.
\end{theorem}

This paper is organized as follows: Section \ref{sec: monotonicity for heat operator} defines the parabolic frequency and proves Theorem \ref{thm: monotonicity of frequency on RFs} and Corollary \ref{cor: backwards uniqueness}; Section \ref{sec: More general parabolic equations} proves Theorem \ref{thm: more general heat operators}, concerning more general parabolic equations.

%% file: 01acknowledgement.tex
The authors are indebted to William Minicozzi for his continual guidance and support, as well as to Tristan Ozuch for useful discussions about Ricci flow. Part of this work was completed while the first author was funded by a National Science Foundation Graduate Research Fellowship.

%% file: 1section.tex
%%%%%%%%%%%%%%%%%%%%%%%%%%%%%%%%%%
\section{The heat equation}\label{sec: monotonicity for heat operator}

To define the parabolic frequency, let $(M^n,g(t))$ for $t\in [t_0,t_1)$ be a Ricci flow evolving by
\begin{equation}\label{eqn: Ricci flow eqn}
    \partial_tg=-2\Ric,
\end{equation}
and let $\tau(t)=t_1-t$ denote backwards time. As in the definition of Almgren's elliptic frequency as well as Poon's parabolic frequency on Euclidean space, fix a point $x_1\in M$ at which the frequency function is centered. Let 
\begin{equation}
    K=(4\pi \tau)^{-\frac{n}{2}}e^{-f}
\end{equation}
be the positive solution of the conjugate heat equation,
\begin{equation}
    \partial_t K=-\Delta K+\Scal K,
\end{equation}
approaching a $\delta$-function centered at $x_1\in M$ as $t\nearrow t_1$. Here and throughout, $\Scal$ denotes the scalar curvature. Note that $K$ is simply the conjugate heat kernel centered at $(x_1,t_1)$.
As a consequence of the definition of $K$, the function $f$ evolves by the following backwards heat-type equation
\begin{equation}\label{eqn: evolution of f}
    \partial_tf = -\Delta f-\Scal +|\nabla f|^2+\frac{n}{2\tau},
\end{equation}
which is the same equation Perelman used in the proof of the monotonicity of his $\mathcal{W}$-functional along Ricci flow \cite[\S 3]{P}.
For the parabolic frequency on Ricci flows, the conjugate heat kernel is the correct generalization of the backwards heat kernel on Euclidean space, which Poon used in his definition of parabolic frequency \cite{P}.

Assume that $(M^n,g(t),f(t))$ satisfies a time-dependent Bakry-\'Emery Ricci curvature bound
\begin{equation}\label{eqn: BE curvature bound}
    \Ric_f\leq \frac{\kappa}{2\tau}g,
\end{equation}
for some constant $\kappa=\kappa(t)$. If $M^n$ is compact, such a function $\kappa$ always exists and depends only on the geometry of the flow at each time \cite{Hu}. 
Note that if $(M^n,g(t),f(t))$ is a subsolution of the shrinking Ricci soliton equation (compact or not), then (\ref{eqn: BE curvature bound}) holds with $\kappa \equiv 1$.

Let 
\begin{equation}
    d\nu=KdV=(4\pi \tau)^{-\frac{n}{2}}e^{-f}\,dV
\end{equation}
denote the corresponding conjugate heat kernel measure and let $\dL_f$ be the corresponding drift Laplacian, defined by
\begin{equation}\label{eqn: drift Laplacian}
    \dL_{f}u=\Delta u-\langle \nabla f,\nabla u\rangle = e^{f}\,\Div\,(e^{-f}\nabla u).
\end{equation}
The drift Laplacian $\dL_f$ is self-adjoint on the weighted Sobolev space $W^{1,2}(d\nu)$, satisfying
\begin{align}\label{eqn: self-adjointness of drift Laplacian}
    \int_M(\dL_{f} u)v \, d\nu
    =-\int_M\langle \nabla u,\nabla v\rangle\,d\nu,
    % &=(4\pi \tau)^{-\frac{n}{2}}\int_Me^f \Div(e^{-f}\nabla u)v \, e^{-f}dV \\
    % &=-(4\pi \tau)^{-\frac{n}{2}}\int_Me^{-f}\langle \nabla u ,\nabla v \rangle\, dV  \nonumber\\
    % &=(4\pi \tau)^{-\frac{n}{2}}\int_Mu\,\Div(e^{-f}\nabla v )\, dV  \nonumber\\
    % &=(4\pi \tau)^{-\frac{n}{2}}\int_Mu\,e^{f}\Div(e^{-f}\nabla v )\, e^{-f}dV  \nonumber\\
    % &=\int_Mu(\dL_fv) \, d\nu. \nonumber
\end{align}
which follows directly from the definition of $\dL_f$ and $d\nu$ via integration by parts.

Differential operators like $\dL_f$ which are naturally associated to weighted measures, have proven to be invaluable in analysis and geometry; they have been used 
in Ricci and mean curvature flow to analyze solitons \cite{CM2, CZ, MW}, by Witten in his study of Morse theory \cite{W}, and recently to generalize the positive mass theorem to weighted manifolds and derive a new monotonicity formula for spinors in the Ricci flow \cite{BO}. See \cite{BH} for a study of eigenvalue estimates of drift Laplacians.

The crucial property satisfied by the conjugate heat kernel measure is its evolution along the Ricci flow. 
Since under the Ricci flow equation (\ref{eqn: Ricci flow eqn}) the volume form $dV$ evolves by
\begin{equation}
    \partial_t(dV)=-\Scal\, dV,
\end{equation}
the conjugate heat kernel measure evolves by,
\begin{equation}\label{eqn: evolution of conjugate heat kernel measure}
    \partial_t(d\nu)
    % =(\partial_t K)dV + K\partial_t(dV)
    % =(-\Delta K+\Scal K)dV-K\Scal dV
    =-(\Delta K)dV
    =-\frac{\Delta K}{K}d\nu.
\end{equation}

Finally, the parabolic frequency can be defined: given $[a,b]\subset [t_0,t_1)$ and a function $u:M\times [a,b]\to \R$ with
$u,\partial_tu \in W^{2,2}(d\nu)$ for each time $t\in [a,b]$, define
\begin{align}
    I(t)
    &=\int_M|u|^2\;d\nu, \label{eqn: defn of I}\\
    D(t)
    &=-\tau\int_M|\nabla u|^2 \;d\nu
    =\tau\int_M\langle u,\dL_{f}u\rangle \; d\nu, \label{eqn: defn of D}\\
    U(t)&=e^{\int \frac{1-\kappa(t)}{\tau(t)}dt}\frac{D(t)}{I(t)}. \label{eqn: defn of U}
\end{align}
With this paper's convention, $U$ is always non-positive. The integral $\int \frac{1-\kappa(t)}{\tau(t)}dt$ in the exponential term is an error term which vanishes if $(M^n,g(t),f(t))$ is a subsolution of the shrinking Ricci soliton equation, since in this case, one may choose $\kappa \equiv 1$.

The following lemma is used in the computation of the derivative of $D$.

\begin{lemma}\label{lem: weighted L2 norm of Hessian}
If $u\in W^{2,2}(d\nu)$ for each time, then 
\begin{equation}
    \int_M |\Hess_u|^2\,d\nu=
    \int_M\left(|\dL_f u|^2-\Ric_f(\nabla u,\nabla u)\right)d\nu.
\end{equation}
\end{lemma}

\begin{proof}
The drift Bochner formula gives
\begin{align}
    \frac{1}{2}\dL_f|\nabla u|^2
    &=|\Hess_u|^2+\langle \nabla u,\nabla \dL_f u\rangle +\Ric_f(\nabla u,\nabla u).
\end{align}
Self-adjointness of $\dL_f$ with respect to the conjugate heat kernel measure and integration by parts give
\begin{align}
    0
    &=\frac{1}{2}\int_M(\dL_f|\nabla u|^2) \, d\nu \label{ques: 1} \\
    &=\int_M\left(|\Hess_u|^2+\langle \nabla u,\nabla \dL_f u\rangle +\Ric_f(\nabla u,\nabla u)\right)d\nu\nonumber\\
    &=\int_M\left(|\Hess_u|^2-|\dL_f u|^2 +\Ric_f(\nabla u,\nabla u)\right)d\nu.\nonumber
\end{align}
\end{proof}

\begin{proof}
[Proof of Theorem \ref{thm: monotonicity of frequency on RFs}]
% (Of Theorem \ref{thm: monotonicity of frequency on RFs})
The monotonicity of $U$ is proven by computing $I'$ and $D'$, and then using the Cauchy-Schwarz inequality to deduce that the difference $I^2U'=ID'-DI'$ is nonnegative.

On a Ricci flow evolving by (\ref{eqn: Ricci flow eqn}), the norm of the gradient of a function $u(x,t)$ evolves by (see \cite[Lemma 4.8]{CM3} for a computation):
\begin{equation}\label{eqn: evolution of gradient squared}
    \partial_t|\nabla u|^2=2\Ric(\nabla u,\nabla u)+2\langle \nabla u,\nabla \partial_tu\rangle.
\end{equation}
Combined with the Bochner formula, it follows that if $u$ solves the heat equation, then
\begin{align}
    (\partial_t-\Delta)|\nabla u|^2
    &=2\Ric(\nabla u,\nabla u)+2\langle \nabla u,\nabla \partial_t u\rangle \nonumber\\
    &\qquad \qquad\qquad -2|\Hess_u|^2-2\langle \nabla u,\nabla \Delta u\rangle -2\Ric(\nabla u,\nabla u) \nonumber\\
    &=-2|\Hess_u|^2. \label{eqn: heat operator on gradient squared}
\end{align}
Hence if $u$ solves the heat equation, then (\ref{eqn: evolution of conjugate heat kernel measure}), integration by parts, and the formula $\Delta u^2=2u\Delta u+2|\nabla u|^2$ imply
\begin{align}
    I'(t)
    &=\int_M\left(2u \partial_t u-u^2\frac{\Delta K}{K}\right)d\nu \\
    &=\int_M\left(2u \Delta u-2u\Delta u-2|\nabla u|^2\right)d\nu \nonumber\\
    &=-2\int_M|\nabla u|^2 \,d\nu \nonumber\\
    &=\frac{2}{\tau(t)}D(t).\nonumber
\end{align}
Hence $I'(t)=\frac{2D(t)}{\tau(t)}$. To compute the derivative of $D$, use the evolution equation of the conjugate heat kernel measure (\ref{eqn: evolution of conjugate heat kernel measure}), integration by parts, equation (\ref{eqn: heat operator on gradient squared}) and Lemma \ref{lem: weighted L2 norm of Hessian} to obtain
\begin{align}
    D'(t)
    &=-\tau\int_M\left(\frac{\tau'}{\tau}|\nabla u|^2 + \partial_t |\nabla u|^2-|\nabla u|^2\frac{\Delta K}{K}\right)d\nu \\
    &=-\tau\int_M\left(-\frac{1}{\tau}|\nabla u|^2 + (\partial_t -\Delta)|\nabla u|^2\right)d\nu
    \nonumber\\
    &=-\tau\int_M\left(-\frac{1}{\tau}|\nabla u|^2-2|\Hess_u|^2\right)d\nu \nonumber\\
    &=-\tau\int_M\left(-\frac{1}{\tau}|\nabla u|^2-2|\dL_fu|^2+2\Ric_f(\nabla u,\nabla u)\right)d\nu \nonumber\\
    &=2\tau\int_M|\dL_fu|^2\,d\nu-2\tau\int_M\left(\Ric_f(\nabla u,\nabla u)-\frac{1}{2\tau}|\nabla u|^2\right)d\nu.\nonumber
\end{align}
Applying the Bakry-\'Emery curvature bound (\ref{eqn: BE curvature bound}) to the last term above gives
\begin{align}
    D'(t)
    &\geq 2\tau\int_M|\dL_fu|^2\,d\nu-(\kappa-1)\int_M|\nabla u|^2\,d\nu \\
    &= 2\tau\int_M|\dL_fu|^2\,d\nu+\frac{(\kappa-1)}{\tau}D(t).\nonumber
\end{align}
Using this bound for $D'$ allows the derivative of the frequency function to be bounded as follows:
\begin{align}
    I^2U'
    &=I^2\left(\frac{1-\kappa}{\tau}e^{\int \frac{1-\kappa}{\tau}}\frac{D}{I}+e^{\int \frac{1-\kappa}{\tau}}\frac{ID'-I'D}{I^2}\right) \\
    &=e^{\int \frac{1-\kappa}{\tau}}\left(\frac{1-\kappa}{\tau}ID+ID'-\frac{2}{\tau}D^2\right) \nonumber\\
    &\geq e^{\int \frac{1-\kappa}{\tau}}\left(\frac{1-\kappa}{\tau}ID+2\tau I\int_M|\dL_fu|^2\,d\nu+\frac{\kappa-1}{\tau}ID-\frac{2}{\tau}D^2\right) \nonumber\\
    &= e^{\int \frac{1-\kappa}{\tau}}\left(2\tau I\int_M|\dL_fu|^2\,d\nu-\frac{2}{\tau}D^2\right) \nonumber\\
    &=e^{\int \frac{1-\kappa}{\tau}}\left(2\tau\left(\int_M|u|^2\,d\nu\right)\left(\int_M|\dL_fu|^2\,d\nu\right)-\frac{2}{\tau}\left(\tau\int_M(u\dL_fu)\,d\nu\right)^2\right) \nonumber\\
    &=2\tau e^{\int \frac{1-\kappa}{\tau}}\left(\left(\int_M|u|^2\,d\nu\right)\left(\int_M|\dL_fu|^2\,d\nu\right)-\left(\int_M(u\dL_fu)\,d\nu\right)^2\right) \nonumber\\
    &\geq 0,\nonumber
\end{align}
where the last inequality is the Cauchy-Schwarz inequality. If $U'(t)=0$, then equality in the Cauchy-Schwarz inequality implies $\dL_fu=c(t)u$. The function $c(t)$ can be determined by noting that if $\dL_fu=c(t)u$, then
\begin{align}
    I(t)U(t)
    &=e^{\int \frac{1-\kappa}{\tau}}D(t) \\
    &=\tau e^{\int \frac{1-\kappa}{\tau}}\int_M\langle u,\dL_fu\rangle \,d\nu \nonumber\\
    &=\tau e^{\int \frac{1-\kappa}{\tau}} c(t)\int_M|u|^2 \,d\nu \nonumber\\
    &=\tau e^{\int \frac{1-\kappa}{\tau}} c(t)I(t),\nonumber
\end{align}
which implies that $c(t)=e^{-\int \frac{1-\kappa}{\tau}}\frac{U(t)}{\tau}$.
\end{proof}

\begin{proof}
[Proof of Corollary \ref{cor: backwards uniqueness}]
% (Of Corollary \ref{cor: backwards uniqueness})
The proof of Theorem \ref{thm: monotonicity of frequency on RFs} shows that $I'=\frac{2}{\tau}D$. Hence 
\begin{equation}
    \log(I)'
    =\frac{I'}{I}
    =\frac{2}{\tau}\frac{D}{I}
    =2e^{-\int \frac{1-\kappa}{\tau}}\frac{U}{\tau}. \label{eqn: log-I-prime}
\end{equation}
Integration and monotonicity of $U$ yields
\begin{equation}
    \log I(b)-\log I(a)
    \geq 2\int_a^b e^{-\int_a^t \frac{1-\kappa(s)}{\tau(s)}ds}\frac{U(t)}{\tau(t)}dt \label{eqn: difference of log(I)}
    \geq 2U(a)\int_a^be^{-\int_a^t \frac{1-\kappa(s)}{\tau(s)}ds}\frac{1}{\tau(t)}dt.
\end{equation}
Note that the latter integral is finite since $0<t_1-b\leq \tau\leq t_1-a$ on $[a,b]\subset [t_0,t_1)$.
Thus, exponentiating gives
\begin{equation}
    I(b)\geq I(a)\exp\left(2U(a)\int_a^be^{-\int_a^t \frac{1-\kappa(s)}{\tau(s)}ds}\tau^{-1}(t)dt\right), \label{eqn: backward uniqueness}
\end{equation}
from which backwards uniqueness follows immediately.
\end{proof}

%% file: 2section.tex
%%%%%%%%%%%%%%%%%%%%%%%%%%%%%%%%%%
\section{More general parabolic equations}\label{sec: More general parabolic equations}

This section extends the results of the previous section to solutions of more general parabolic equations on Ricci flows; see \cite{P, ESS, CM3} for the static case.

In this section, again assume that $(M^n,g(t))$, for $t\in [t_0,t_1)$, is a Ricci flow evolving by (\ref{eqn: Ricci flow eqn}) with a bound on curvature and its first derivative. As before, let $\tau(t)=t_1-t$ denote backwards time and let
\begin{equation}
    K=(4\pi \tau)^{-\frac{n}{2}}e^{-f}
\end{equation}
be the conjugate heat kernel centered at some point $(x_1,t_1)$ in spacetime. Again assume there exist time-dependent constants $\kappa=\kappa(t)$ such that on $M\times [t_0,t_1)$ the bound (\ref{eqn: BE curvature bound}) holds. Let $d\nu =K\, dV$ denote the corresponding conjugate heat kernel measure and $\dL_f$ is the corresponding drift Laplacian, which is self-adjoint on the weighted Sobolev space $W^{1,2}(d\nu)$.

Given $[a,b]\subset [t_0,t_1)$ and a function $u:M\times [a,b]\to \R$ with $u,\partial_tu \in W^{2,2}(d\nu)$ for each time $t\in [a,b]$, define the functions $I,D,U$ as in (\ref{eqn: defn of I})-(\ref{eqn: defn of U}).

\begin{theorem}
\label{thm: quantitative derivative bounds, general case}
Let $u$ satisfy 
\begin{equation}
    |(\partial_t-\Delta) u|\leq C(t)(|\nabla u|+|u|) \label{eqn: general operator}
\end{equation}
along the Ricci flow $(M,g(t))$ for $t\in [a,b]$. Then
\begin{align}
    \log(I)' &\geq (2+C)e^{-\int \frac{1-\kappa}{\tau}}\frac{U}{\tau}-3C \\
    U' &\geq C^2(U-\tau) \\
    C^2 &\geq \log(\tau(a)-U)'.
\end{align}
\end{theorem}

\begin{proof}
Differentiating $I(t)$ and using the elementary inequality $2ab\leq a^2+b^2$ gives
\begin{align}
    I'
    &=\int_M\left(2u \partial_t u-u^2\frac{\Delta K}{K}\right)d\nu \\
    &=\int_M\left(2u \partial_t u-2u\Delta u-2|\nabla u|^2\right)d\nu \nonumber\\
    &=\frac{2}{\tau}D+2\int_Mu(\partial_t-\Delta)u \,d\nu \nonumber\\
    &\geq \frac{2}{\tau}D-2C\int_M|u|(|\nabla u|+|u|) \,d\nu \nonumber\\
    &= \frac{2}{\tau}D-2CI-2C\int_M|u||\nabla u| \,d\nu \nonumber\\
    &\geq \frac{2}{\tau}D-3CI+\frac{C}{\tau}D \nonumber\\
    &= \frac{2+C}{\tau}D-3CI \nonumber\\
    &= (2+C)e^{-\int \frac{1-\kappa}{\tau}}I\frac{U}{\tau}-3CI,\nonumber
\end{align}
which proves the first claim.

On a Ricci flow evolving by (\ref{eqn: Ricci flow eqn}), the norm of the gradient of a function $u(x,t)$ evolves by (\ref{eqn: evolution of gradient squared}).
Combined with the Bochner formula, this implies that
\begin{align}\label{eqn: heat operator on gradient squared (general case)}
    (\partial_t-\Delta)|\nabla u|^2
    &=2\Ric(\nabla u,\nabla u)+2\langle \nabla u,\nabla \partial_t u\rangle \\
    &\qquad \qquad\qquad -2|\Hess_u|^2-2\langle \nabla u,\nabla \Delta u\rangle -2\Ric(\nabla u,\nabla u) \nonumber\nonumber\\
    &=-2|\Hess_u|^2+2\langle \nabla u,\nabla (\partial_t-\Delta)u\rangle.\nonumber
\end{align}
Write $D$ as 
\begin{equation}\label{eqn: D rewritten}
    D=\tau\int_Mu\left(\dL_f+\frac{1}{2}(\partial_t-\Delta)\right)u\,d\nu - \frac{\tau}{2}\int_Mu(\partial_t-\Delta)u\,d\nu.
\end{equation}
Then by the integration by parts formula (\ref{eqn: self-adjointness of drift Laplacian}) for $\dL_f$, it follows that
\begin{align}
    I'(t)
    &=\int_M\left(2u \partial_t u-u^2\frac{\Delta K}{K}\right)d\nu \nonumber\\
    &=\int_M\left(2u \partial_t u-2u\Delta u-2|\nabla u|^2\right)d\nu \nonumber\\
    &=2\int_Mu(\dL_f+\partial_t-\Delta)u \,d\nu \nonumber\\
    &=2\int_Mu\left(\dL_f+\frac{1}{2}(\partial_t-\Delta)\right)u \,d\nu
    +\int_Mu(\partial_t-\Delta)u \,d\nu. \label{eqn: I' rewritten}
\end{align}
Combining (\ref{eqn: I' rewritten}) and (\ref{eqn: D rewritten}) yields
\begin{equation}\label{eqn: ID'}
    I'D=2\tau \left(\int_Mu\left(\dL_f+\frac{1}{2}(\partial_t-\Delta)\right)u\,d\nu\right)^2
    -\frac{\tau}{2}\left(\int_Mu(\partial_t-\Delta)u\,d\nu\right)^2.
\end{equation}
By (\ref{eqn: heat operator on gradient squared (general case)}),
\begin{align}
    D'(t)
    &=-\tau\int_M\left(\frac{\tau'}{\tau}|\nabla u|^2 + \partial_t |\nabla u|^2-|\nabla u|^2\frac{\Delta K}{K}\right)d\nu \\
    &=-\tau\int_M\left(-\frac{1}{\tau}|\nabla u|^2 + (\partial_t -\Delta)|\nabla u|^2\right)d\nu \nonumber\\
    &=-\tau\int_M\left(-\frac{1}{\tau}|\nabla u|^2-2|\Hess_u|^2+2\langle \nabla u,\nabla(\partial_t-\Delta)u\rangle \right)d\nu.\nonumber
\end{align}
Applying the integration by parts formula (\ref{eqn: self-adjointness of drift Laplacian}) for $\dL_f$ to the last term of the integrand, as well as Lemma \ref{lem: weighted L2 norm of Hessian}, gives
\begin{align}
    D'(t)
    &=-\tau\int_M\left(-\frac{1}{\tau}|\nabla u|^2-2|\dL_fu|^2+2\Ric_f(\nabla u,\nabla u)-2(\dL_fu)(\partial_t-\Delta)u\right)d\nu \\
    &=2\tau\int_M(|\dL_fu|^2+(\dL_fu)(\partial_t-\Delta)u)\,d\nu-2\tau\int_M\left(\Ric_f(\nabla u,\nabla u)-\frac{1}{2\tau}|\nabla u|^2\right)d\nu, \nonumber\\
    &\geq 2\tau\int_M(|\dL_fu|^2+(\dL_fu)(\partial_t-\Delta)u)\,d\nu-(\kappa-1)\int_M|\nabla u|^2\,d\nu, \nonumber\\
    &= \frac{(\kappa-1)}{\tau}D(t) 
    +2\tau \int_M\left(\left|\left(\dL_f+\frac{1}{2}(\partial_t-\Delta)\right)u\right|^2
    -\frac{1}{4}|(\partial_t-\Delta)u|^2\right)d\nu,\nonumber
\end{align}
where the divergence theorem, integration by parts, Bochner's formula, and the Bakry-\'Emery curvature bound (\ref{eqn: BE curvature bound}), and completion of the square have been used.
Using this bound for $D'$, (\ref{eqn: ID'}), the Cauchy-Schwarz inequality, and the bound $(a+b)^2\leq 2(a^2+b^2)$ gives
\begin{align}
    I^2U'
    &=e^{\int \frac{1-\kappa}{\tau}}\left(\frac{1-\kappa}{\tau}ID+ID'-I'D\right) \\
    &\geq e^{\int \frac{1-\kappa}{\tau}}\left[\frac{1-\kappa}{\tau}ID+\frac{\kappa-1}{\tau}ID \right. \nonumber\\
    & \left. \qquad +2\tau \left[ I\left(\int_M\left|\left(\dL_f+\frac{1}{2}(\partial_t-\Delta)\right)u\right|^2d\nu \right) 
    -\left(\int_Mu\left(\dL_f+\frac{1}{2}(\partial_t-\Delta)\right)u\,d\nu\right)^2\right]  \right. \nonumber \\
    & \left. \qquad +\frac{\tau}{2}\left(\int_Mu(\partial_t-\Delta )u\,d\nu\right)^2
    -\frac{\tau}{2}I\int_M|(\partial_t-\Delta)u|^2\,d\nu \right] \nonumber \\
    &\geq -\frac{\tau}{2}e^{\int \frac{1-\kappa}{\tau}}I\int_M|(\partial_t-\Delta)u|^2\,d\nu\nonumber\\
    &\geq -\frac{\tau}{2}e^{\int \frac{1-\kappa}{\tau}}C^2I\int_M(|\nabla u|+|u|)^2\,d\nu\nonumber\\
    &\geq e^{\int \frac{1-\kappa}{\tau}}C^2I(D-\tau I).\nonumber
\end{align}
Therefore, using that $\kappa \geq 1$ and $\tau'<0$,
\begin{equation}
    U'
    \geq C^2(U-\tau e^{\int \frac{1-\kappa}{\tau}})
    \geq C^2(U-\tau)
    \geq C^2(U-\tau(a)),
\end{equation}
which proves the second and third claims in the Theorem.
\end{proof}

\begin{corollary}
[Theorem \ref{thm: more general heat operators} restated]
Let $u:M\times [a,b]\to \R$ satisfy $|(\partial_t-\Delta) u|\leq C(t)(|\nabla u|+|u|)$ along the Ricci flow $(M,g(t))$. Then
\begin{align}\label{eqn: log(I) bound}
    I(b)
    &\geq I(a)\exp\left[
    \left(2+\sup_{[a,b]}C\right)
    \left((U(a)-\tau(a))\exp\left(\int_a^bC^2 \,dt\right)+\tau(a)\right)
    \right. \\ &\left. \qquad\qquad\qquad\qquad\qquad\qquad
    \times\left(\int_a^b e^{-\int_a^t \frac{1-\kappa(s)}{\tau(s)}ds}\tau^{-1}(t)\,dt\right) 
    -3(b-a)\sup_{[a,b]}C
    \right]. \nonumber
\end{align}
In particular, if $u(\cdot,b)=0$, then $u\equiv 0$ for all $t\in [a,b]$.
\end{corollary}

\begin{proof}
Integrating the first claim in Theorem \ref{thm: quantitative derivative bounds, general case} gives
\begin{align}\label{eqn: intermediate log(I) bound}
    \log(I(b))-\log(I(a))
    &\geq \int_a^b(2+C)e^{-\int_a^t \frac{1-\kappa(s)}{\tau(s)}ds}\frac{U(t)}{\tau(t)}dt-3\int_a^bC\,dt \\
    &\geq \left(2+\sup_{[a,b]}C\right)\int_a^be^{-\int_a^t \frac{1-\kappa(s)}{\tau(s)}ds}\frac{U(t)}{\tau(t)}\,dt-3(b-a)\sup_{[a,b]}C. \nonumber
\end{align}
Integrating the third claim in Theorem \ref{thm: quantitative derivative bounds, general case} gives, for any $t\in [a,b]$,
\begin{align}
    \log(\tau(a)-U(t))
    \leq \log(\tau(a)-U(a))+\int_a^tC^2\,ds \label{eqn: tight bound of U} \\
    \leq \log(\tau(a)-U(a))+\int_a^bC^2\,ds.\nonumber
\end{align}
Therefore,
\begin{equation}
    U(t)\geq (U(a)-\tau(a))\exp\left(\int_a^bC^2\,ds\right)+\tau(a). 
\end{equation}
Inserting this lower bound into (\ref{eqn: intermediate log(I) bound}) and exponentiating proves the first part of the Corollary. The second part follows immediately from the first. (Note as in the previous section that the integral $\int_a^b e^{-\int_a^t \frac{1-\kappa(s)}{\tau(s)}ds}\tau^{-1}(t)\,dt$ is finite since $\tau$ is uniformly positive on $[a,b]\subset [t_0,t_1)$.)
\end{proof}

%% file: 3appendix.tex
%%%%%%%%%%%%%%%%%%%%%%%%%%%%%%%%%%
\begin{appendices}
%%%%%%%%%%%%%%%%%%%%%%%%%%%%%%%%%%
\section{The Gaussian soliton}\label{sec: examples}

% \paragraph{Shrinking solitons.}
Consider a shrinking Ricci soliton flow $(M^n,g(t),f(t))_{t<0}$ with 
\begin{equation}\label{eqn: shrinker equation}
    \Ric + \Hess_f=\frac{1}{2\tau}g,
\end{equation}
where $\tau(t)=-t$ denotes backwards time. Then $f$ evolves by (\ref{eqn: evolution of f}) and consequently the function $(4\pi \tau)^{-\frac{n}{2}}e^{-f}$ satisfies the conjugate heat equation. 
Thus, on a shrinking Ricci soliton, the function $(4\pi \tau)^{-\frac{n}{2}}e^{-f}$ shares many formal properties with the fundamental solution of the conjugate heat equation centered at some point. These properties were used by Colding-Minicozzi \cite{CM3} to define a fruitful notion of parabolic frequency on shrinking Ricci solitons, and inspired the results of this paper.
For a detailed analysis of the fundamental solution of the conjugate heat equation on shrinking Ricci solitons, and in particular its relationship with the soliton potential $f$, see \cite{LW2}.

% \paragraph{Shrinking Gaussian soliton.}
As a special case, consider the shrinking Gaussian soliton flow $(\R^n,g(t),f(t))_{t<0}$, with
\begin{equation}
    g(t)=g_{\euc}, \qquad\qquad \qquad f(x,t)=\frac{|x|^2}{4\tau},
\end{equation}
where $\tau(t)=-t$ again denotes backwards time. See \cite[\S 4.2]{CLN} for a detailed description of the shrinking Gaussian soliton. 
The corresponding drift Laplacian is given by
\begin{equation}
    \dL_f=\Delta-\frac{1}{2\tau}\langle x,\nabla(\cdot)\rangle,
\end{equation}
By Proposition \ref{prop: Spectrum of shrinking Gaussian drift Laplacian}, its spectrum is given by
\begin{equation}\label{eqn: Spectrum of shrinking Gaussian drift Laplacian with time}
    \Spec\left(\dL_{\frac{|x|^2}{4\tau}}\right)=\left\{-\frac{k}{2\tau}\; \bigg \vert \; k\in \N\right\},
\end{equation}
with the $(-\frac{k}{2\tau})$-eigenfunctions being polynomials of degree $k$.

Below, the equality case of Theorem \ref{thm: monotonicity of frequency on RFs} is analyzed on the shrinking Gaussian soliton flow.
That is, the solutions of the heat equation on $\R^n$ for which the frequency is constant are determined, and the value of their frequency is computed. 
Let $u:\R^n\times (-\infty,0)\to \R$ solve the heat equation along the shrinking Gaussian soliton flow and suppose that $U(t)=U(-1)$ is constant. Since $\kappa$ can be chosen to equal 1 for shrinking soliton flows, the equality statement in Theorem \ref{thm: monotonicity of frequency on RFs} implies that 
\begin{equation}
    \dL_fu=e^{-\int \frac{1-\kappa}{\tau}}\frac{U(t)}{\tau(t)}u=\frac{U(-1)}{\tau(t)}u,
\end{equation}
for all $t<0$. By (\ref{eqn: Spectrum of shrinking Gaussian drift Laplacian with time}), it follows that there exists $k\in \N$ such that,
\begin{equation}
    U(t)\equiv -\frac{k}{2},
\end{equation}
and that $u$ is a polynomial of degree $k$. Hence $u$ is a caloric polynomial of degree $k$.

\subsection{Spectral theory of the Ornstein-Uhlenbeck operator}

The drift Laplacian $\dL_f$ associated with the shrinking Gaussian soliton is also called the Ornstein-Uhlenbeck operator; this appendix computes the spectrum of this operator.

\begin{proposition}
[Spectrum of shrinking Gaussian drift Laplacian]\label{prop: Spectrum of shrinking Gaussian drift Laplacian}
Let $\tau >0$. The spectrum of the shrinking Gaussian drift Laplacian $\dL_{\frac{|x|^2}{4\tau}}$ on $L^2(\R^n,e^{-\frac{|x|^2}{4\tau}}dx)$ is given by
\begin{equation}
    \Spec\left(\dL_{\frac{|x|^2}{4\tau}}\right)=\left\{-\frac{k}{2\tau}\; \bigg \vert \; k\in \N\right\}.
\end{equation}
Moreover, if $u\in L^2(e^{-\frac{|x|^2}{4\tau}}dx)$ is an eigenfunction with $\dL_{\frac{|x|^2}{4\tau}}u=-\frac{k}{2\tau} u$, then $u$ is a polynomial of degree $k$.
\end{proposition}

\begin{proof}
The result is standard, see e.g. \cite[Prop. 3.1]{MPP} for a proof using parabolic methods. For the convenience of the reader, an elementary proof is given here based on the commutator formula 
\begin{equation}
\left[\dL_{\frac{|x|^2}{4\tau}},\partial_i\right]u=\frac{1}{2\tau}\partial_iu,
\end{equation}
which may be verified as follows: setting $f=\frac{|x|^2}{4\tau}$ for the remainder of this proof, it follows that $\nabla f=\frac{x}{2\tau}$ and hence
\begin{align}
    \dL_f\partial_iu
    &=\Delta \partial_iu-\frac{1}{2\tau}\langle x,\nabla \partial_iu\rangle \\
    &= \partial_i\Delta u-\frac{1}{2\tau}\langle x_ie_i,(\partial_j \partial_iu)e_j\rangle \nonumber\\
    &= \partial_i\Delta u-\partial_i\left(\frac{1}{2\tau}\langle x_ie_i,(\partial_ju)e_j\rangle\right) +\frac{1}{2\tau}\langle e_i,(\partial_ju)e_j\rangle \nonumber\\
    &= \partial_i\dL_fu +\frac{1}{2\tau}\partial_iu.\nonumber
\end{align}

In particular, if $u$ is a $\dL_f$-eigenfunction with $\dL_f u=\lambda u$, the commutator formula implies that
\begin{equation}
    \dL_f(\partial_iu)=\left(\frac{1}{2\tau}+\lambda\right)\partial_iu,
\end{equation}
and iteration gives, for any multi-index $\alpha \in \N^n$, that
\begin{equation}
    \dL_fu_{\alpha}=\left(\frac{|\alpha|}{2\tau}+\lambda\right)u_{\alpha},
\end{equation}
where $u_{\alpha}=\frac{\partial^{|\alpha|}u}{\partial x^{\alpha}}$. In other words, $u_{\alpha}$ is an $\dL_f$-eigenfunction with eigenvalue $\frac{|\alpha|}{2\tau}+\lambda$.
Choosing $|\alpha|\geq -2\tau \lambda$, forces the eigenvalue $\frac{|\alpha|}{2\tau}+\lambda$ to be nonnegative. The integration by parts formula (\ref{eqn: self-adjointness of drift Laplacian}) then gives that
\begin{equation}
    0
    \leq \left(\frac{|\alpha|}{2\tau}+\lambda\right)\int_{\R^n}|u_{\alpha}|^2 \,e^{-f}dx
    =\int_{\R^n}(\dL_fu_{\alpha})u_{\alpha}\,e^{-f}dx
    =-\int_{\R^n}|\nabla u_{\alpha}|^2\,e^{-f}dx
    \leq 0.
\end{equation}
Hence $\nabla u_{\alpha}=0$, meaning $u_{\alpha}$ is constant. Moreover, if $u_{\alpha}$ is not identically zero, then $\lambda = -\frac{|\alpha|}{2\tau}$.
The fundamental theorem of calculus then implies that $u$ is necessarily a polynomial of degree $|\alpha|$. This proves the inclusion $\Spec\left(\dL_f\right)\subset \{-\frac{k}{2\tau} \mid  k\in \N\}$.

The opposite inclusion, $\{-\frac{k}{2\tau} \mid  k\in \N\}\subset \Spec\left(\dL_f\right)$, is shown using the Lemma below: by Lemma \ref{lem: 1-dimensional shrinker spectrum}, there exists a Hermite polynomial $u:\R\to \R$ with $\dL_{x_1^2/4\tau}u=-\frac{k}{2\tau}u$. Extending $u$ trivially to a function on $\R^n$ then yields $\dL_fu=-\frac{k}{2\tau}u$.
\end{proof}

\begin{lemma}
[1-dimensional spectrum]\label{lem: 1-dimensional shrinker spectrum}
Let $\tau >0$. The spectrum of the 1-dimensional shrinking Gaussian drift Laplacian $\dL_{\frac{x^2}{4\tau}}$ on $L^2(\R,e^{-\frac{x^2}{4\tau}}dx)$ is given by
\begin{equation}
    \Spec\left(\dL_{\frac{x^2}{4\tau}}\right)=\left\{-\frac{k}{2\tau}\; \bigg \vert \; k\in \N\right\},
\end{equation}
and the eigenfunctions are precisely the Hermite polynomials.
\end{lemma}

\begin{proof}
    Here the case $\tau =1$ is proven; the general case follows by scaling. In 1 dimension, the drift Laplace eigenvalue equation (for $\tau=1$) reduces to the following ODE:
    \[
    u''(x) - \dfrac{x}{2}u'(x) - \lambda u(x) = 0.
    \]
    Setting $v(x) = u(2x)$, the ODE is equivalent to Hermite's differential equation:
    \[
    v''(x) - 2xv'(x) - 4\lambda v(x) = 0.
    \]

    Now it is shown that for $\lambda = -\frac{k}{2}$ where $k$ is a non-negative integer, the $k$-th Hermite polynomial is a solution to the ODE. Let $D^k=\frac{d^k}{dx^k}$ and define $v(x) = e^{x^2}D^ke^{-x^2}$, which is the $k$-th Hermite polynomial up to sign. Then
    \begin{align*}
        v'(x) &= e^{x^2}D^{k+1}e^{-x^2} + 2xe^{x^2}D^ke^{-x^2}\\
        v''(x) &= e^{x^2}D^{k+2}e^{-x^2} + 4xe^{x^2}D^{k+1}e^{-x^2} + (4x^2+2)e^{x^2}D^ke^{-x^2}.
    \end{align*}
    Hence,
    \begin{align*}
        v''(x) - 2xv'(x) + 2kv(x)
        &= e^{x^2}D^{k+2}e^{-x^2} + 2xe^{x^2}D^{k+1}e^{-x^2} + (2k+2)e^{x^2}D^ke^{-x^2}\\
        &= -2e^{x^2}D^{k+1}(xe^{-x^2}) + 2xe^{x^2}D^{k+1}e^{-x^2} + (2k+2)e^{x^2}D^ke^{-x^2} \\
        &= 0,
    \end{align*}
    since $D^{k+1}(xe^{-x^2}) = xD^{k+1}e^{-x^2}+(k+1)D^k e^{-x^2}$. Hence, the $k$-th Hermite polynomial is an eigenfunction when $\lambda = -\frac{k}{2}$.
    To show the Hermite polynomials are the only solutions, the argument from the first part of the previous proof can be applied. 
\end{proof}

\end{appendices}

%% file: 4references.tex
%%%%%%%%%%%%%%%%%%%%%%%%%%%%%%%%%%
\Addresses